\documentclass[11pt]{amsart}

\usepackage{amsmath}
\usepackage{amssymb}
\usepackage{amscd}
\usepackage{color}
\usepackage{hyperref}
\usepackage{mathrsfs}

\usepackage{tikz}

\usepackage{xy}
\xyoption{all}

\newcommand{\nc}{\newcommand}

\renewcommand{\AA}{{\mathbb{A}}}
\nc{\CC}{{\mathbb{C}}}
\nc{\LL}{{\mathbb{L}}}
\nc{\RR}{{\mathbb{R}}}
\nc{\PP}{{\mathbb{P}}}
\nc{\OO}{{\mathbb{O}}}

\nc{\QQ}{{\mathbb{Q}}}
\nc{\ZZ}{{\mathbb{Z}}}

\nc{\cA}{{\mathscr{A}}}
\nc{\cB}{{\mathscr{B}}}
\nc{\cC}{{\mathscr{C}}}
\nc{\cD}{{\mathscr{D}}}
\nc{\cE}{{\mathscr{E}}}
\nc{\cF}{{\mathscr{\rF}}}
\nc{\cG}{{\mathscr{G}}}
\nc{\cH}{{\mathscr{H}}}
\nc{\cI}{{\mathscr{I}}}
\nc{\cJ}{{\mathscr{J}}}
\nc{\cK}{{\mathscr{K}}}
\nc{\cL}{{\mathscr{L}}}
\nc{\cM}{{\mathscr{M}}}
\nc{\cN}{{\mathscr{N}}}
\nc{\cO}{{\mathscr{O}}}
\nc{\cP}{{\mathscr{P}}}
\nc{\cQ}{{\mathscr{Q}}}
\nc{\cR}{{\mathscr{R}}}
\nc{\cS}{{\mathscr{S}}}
\nc{\cT}{{\mathscr{T}}}
\nc{\cU}{{\mathscr{U}}}
\nc{\cV}{{\mathscr{V}}}
\nc{\cW}{{\mathscr{W}}}
\nc{\cX}{{\mathscr{X}}}
\nc{\cY}{{\mathscr{Y}}}
\nc{\cZ}{{\mathscr{Z}}}

\nc{\bA}{{\mathbf{A}}}
\nc{\bB}{{\mathbf{B}}}
\nc{\bC}{{\mathbf{C}}}
\nc{\bD}{{\mathbf{D}}}
\nc{\bE}{{\mathbf{E}}}
\nc{\bF}{{\mathbf{\rF}}}
\nc{\bG}{{\mathbf{G}}}
\nc{\bH}{{\mathbf{H}}}
\nc{\bI}{{\mathbf{I}}}
\nc{\bJ}{{\mathbf{J}}}
\nc{\bK}{{\mathbf{K}}}
\nc{\bL}{{\mathbf{L}}}
\nc{\bM}{{\mathbf{M}}}
\nc{\bN}{{\mathbf{N}}}
\nc{\bO}{{\mathbf{O}}}
\nc{\bP}{{\mathbf{P}}}
\nc{\bQ}{{\mathbf{Q}}}
\nc{\bR}{{\mathbf{R}}}
\nc{\bS}{{\mathbf{S}}}
\nc{\bT}{{\mathbf{T}}}
\nc{\bU}{{\mathbf{U}}}
\nc{\bV}{{\mathbf{V}}}
\nc{\bW}{{\mathbf{W}}}
\nc{\bX}{{\mathbf{X}}}
\nc{\bY}{{\mathbf{Y}}}
\nc{\bZ}{{\mathbf{Z}}}

\nc{\ba}{{\mathbf{a}}}
\nc{\bb}{{\mathbf{b}}}
\nc{\bc}{{\mathbf{c}}}
\nc{\bd}{{\mathbf{d}}}
\nc{\be}{{\mathbf{e}}}
\nc{\bg}{{\mathbf{g}}}
\nc{\bh}{{\mathbf{h}}}
\nc{\bi}{{\mathbf{i}}}
\nc{\bj}{{\mathbf{j}}}
\nc{\bk}{{\mathbf{k}}}
\nc{\bl}{{\mathbf{l}}}
\nc{\bm}{{\mathbf{m}}}
\nc{\bn}{{\mathbf{n}}}
\nc{\bo}{{\mathbf{o}}}
\nc{\bp}{{\mathbf{p}}}
\nc{\bq}{{\mathbf{q}}}
\nc{\br}{{\mathbf{r}}}
\nc{\bs}{{\mathbf{s}}}
\nc{\bt}{{\mathbf{t}}}
\nc{\bu}{{\mathbf{u}}}
\nc{\bv}{{\mathbf{v}}}
\nc{\bw}{{\mathbf{w}}}
\nc{\bx}{{\mathbf{x}}}
\nc{\by}{{\mathbf{y}}}
\nc{\bz}{{\mathbf{z}}}

\nc{\fA}{{\mathfrak{A}}}
\nc{\fB}{{\mathfrak{B}}}
\nc{\fC}{{\mathfrak{C}}}
\nc{\fD}{{\mathfrak{D}}}
\nc{\fE}{{\mathfrak{E}}}
\nc{\fF}{{\mathfrak{\rF}}}
\nc{\fG}{{\mathfrak{G}}}
\nc{\fH}{{\mathfrak{H}}}
\nc{\fI}{{\mathfrak{I}}}
\nc{\fJ}{{\mathfrak{J}}}
\nc{\fK}{{\mathfrak{K}}}
\nc{\fL}{{\mathfrak{L}}}
\nc{\fM}{{\mathfrak{M}}}
\nc{\fN}{{\mathfrak{N}}}
\nc{\fO}{{\mathfrak{O}}}
\nc{\fP}{{\mathfrak{P}}}
\nc{\fQ}{{\mathfrak{Q}}}
\nc{\fR}{{\mathfrak{R}}}
\nc{\fS}{{\mathfrak{S}}}
\nc{\fT}{{\mathfrak{T}}}
\nc{\fU}{{\mathfrak{U}}}
\nc{\fV}{{\mathfrak{V}}}
\nc{\fW}{{\mathfrak{W}}}
\nc{\fX}{{\mathfrak{X}}}
\nc{\fY}{{\mathfrak{Y}}}
\nc{\fZ}{{\mathfrak{Z}}}

\nc{\fa}{{\mathfrak{a}}}
\nc{\fb}{{\mathfrak{b}}}
\nc{\fc}{{\mathfrak{c}}}
\nc{\fd}{{\mathfrak{d}}}
\nc{\fe}{{\mathfrak{e}}}
\nc{\ff}{{\mathfrak{f}}}
\nc{\fg}{{\mathfrak{g}}}
\nc{\fh}{{\mathfrak{h}}}
\nc{\fj}{{\mathfrak{j}}}
\nc{\fk}{{\mathfrak{k}}}
\nc{\fl}{{\mathfrak{l}}}
\nc{\fm}{{\mathfrak{m}}}
\nc{\fn}{{\mathfrak{n}}}
\nc{\fo}{{\mathfrak{o}}}
\nc{\fp}{{\mathfrak{p}}}
\nc{\fq}{{\mathfrak{q}}}
\nc{\fr}{{\mathfrak{r}}}
\nc{\fs}{{\mathfrak{s}}}
\nc{\ft}{{\mathfrak{t}}}
\nc{\fu}{{\mathfrak{u}}}
\nc{\fv}{{\mathfrak{v}}}
\nc{\fw}{{\mathfrak{w}}}
\nc{\fx}{{\mathfrak{x}}}
\nc{\fy}{{\mathfrak{y}}}
\nc{\fz}{{\mathfrak{z}}}

\nc{\sA}{{\mathsf{A}}}
\nc{\sB}{{\mathsf{B}}}
\nc{\sC}{{\mathsf{C}}}
\nc{\sD}{{\mathsf{D}}}
\nc{\sE}{{\mathsf{E}}}
\nc{\sF}{{\mathsf{\rF}}}
\nc{\sG}{{\mathsf{G}}}
\nc{\sH}{{\mathsf{H}}}
\nc{\sI}{{\mathsf{I}}}
\nc{\sJ}{{\mathsf{J}}}
\nc{\sK}{{\mathsf{K}}}
\nc{\sL}{{\mathsf{L}}}
\nc{\sM}{{\mathsf{M}}}
\nc{\sN}{{\mathsf{N}}}
\nc{\sO}{{\mathsf{O}}}
\nc{\sP}{{\mathsf{P}}}
\nc{\sQ}{{\mathsf{Q}}}
\nc{\sR}{{\mathsf{R}}}
\nc{\sS}{{\mathsf{S}}}
\nc{\sT}{{\mathsf{T}}}
\nc{\sU}{{\mathsf{U}}}
\nc{\sV}{{\mathsf{V}}}
\nc{\sW}{{\mathsf{W}}}
\nc{\sX}{{\mathsf{X}}}
\nc{\sY}{{\mathsf{Y}}}
\nc{\sZ}{{\mathsf{Z}}}

\nc{\sa}{{\mathsf{a}}}
\nc{\sd}{{\mathsf{d}}}
\nc{\se}{{\mathsf{e}}}
\nc{\sg}{{\mathsf{g}}}
\nc{\sh}{{\mathsf{h}}}
\nc{\si}{{\mathsf{i}}}
\nc{\sj}{{\mathsf{j}}}
\nc{\sk}{{\mathsf{k}}}
\nc{\sm}{{\mathsf{m}}}
\nc{\sn}{{\mathsf{n}}}
\nc{\so}{{\mathsf{o}}}
\nc{\sq}{{\mathsf{q}}}
\nc{\sr}{{\mathsf{r}}}
\nc{\st}{{\mathsf{t}}}
\nc{\su}{{\mathsf{u}}}
\nc{\sv}{{\mathsf{v}}}
\nc{\sw}{{\mathsf{w}}}
\nc{\sx}{{\mathsf{x}}}
\nc{\sy}{{\mathsf{y}}}
\nc{\sz}{{\mathsf{z}}}

\nc{\oA}{{\overline{A}}}
\nc{\oB}{{\overline{B}}}
\nc{\oC}{{\overline{C}}}
\nc{\oD}{{\overline{D}}}
\nc{\oE}{{\overline{E}}}
\nc{\oF}{{\overline{\rF}}}
\nc{\oG}{{\overline{G}}}
\nc{\oH}{{\overline{H}}}
\nc{\oI}{{\overline{I}}}
\nc{\oJ}{{\overline{J}}}
\nc{\oK}{{\overline{K}}}
\nc{\oL}{{\overline{L}}}
\nc{\oM}{{\overline{M}}}
\nc{\oN}{{\overline{N}}}
\nc{\oO}{{\overline{O}}}
\nc{\oP}{{\overline{P}}}
\nc{\oQ}{{\overline{Q}}}
\nc{\oR}{{\overline{R}}}
\nc{\oS}{{\overline{S}}}
\nc{\oT}{{\overline{T}}}
\nc{\oU}{{\overline{U}}}
\nc{\oV}{{\overline{V}}}
\nc{\oW}{{\overline{W}}}
\nc{\oX}{{\overline{X}}}
\nc{\oY}{{\overline{Y}}}
\nc{\oZ}{{\overline{Z}}}

\nc{\oa}{{\overline{a}}}
\nc{\ob}{{\overline{b}}}
\nc{\oc}{{\overline{c}}}
\nc{\od}{{\overline{d}}}
\nc{\of}{{\overline{f}}}
\nc{\og}{{\overline{g}}}
\nc{\oh}{{\overline{h}}}
\nc{\oi}{{\overline{i}}}
\nc{\oj}{{\overline{j}}}
\nc{\ok}{{\overline{k}}}
\nc{\ol}{{\overline{l}}}
\nc{\om}{{\overline{m}}}
\nc{\on}{{\overline{n}}}
\nc{\oo}{{\overline{o}}}
\nc{\op}{{\overline{p}}}
\nc{\oq}{{\overline{q}}}
\nc{\os}{{\overline{s}}}
\nc{\ot}{{\overline{t}}}
\nc{\ou}{{\overline{u}}}
\nc{\ov}{{\overline{v}}}
\nc{\ow}{{\overline{w}}}
\nc{\ox}{{\overline{x}}}
\nc{\oy}{{\overline{y}}}
\nc{\oz}{{\overline{z}}}

\nc{\tA}{{\tilde{A}}}
\nc{\tB}{{\tilde{B}}}
\nc{\tC}{{\tilde{C}}}
\nc{\tD}{{\tilde{D}}}
\nc{\tE}{{\tilde{E}}}
\nc{\tF}{{\tilde{\rF}}}
\nc{\tG}{{\tilde{G}}}
\nc{\tH}{{\tilde{H}}}
\nc{\tI}{{\tilde{I}}}
\nc{\tJ}{{\tilde{J}}}
\nc{\tK}{{\tilde{K}}}
\nc{\tL}{{\tilde{L}}}
\nc{\tM}{{\tilde{M}}}
\nc{\tN}{{\tilde{N}}}
\nc{\tO}{{\tilde{O}}}
\nc{\tP}{{\tilde{P}}}
\nc{\tQ}{{\tilde{Q}}}
\nc{\tR}{{\tilde{R}}}
\nc{\tS}{{\tilde{S}}}
\nc{\tT}{{\tilde{T}}}
\nc{\tU}{{\tilde{U}}}
\nc{\tV}{{\tilde{V}}}
\nc{\tW}{{\tilde{W}}}
\nc{\tX}{{\tilde{X}}}
\nc{\tY}{{\tilde{Y}}}
\nc{\tZ}{{\tilde{Z}}}

\nc{\ta}{{\tilde{a}}}
\nc{\tb}{{\tilde{b}}}
\nc{\tc}{{\tilde{c}}}
\nc{\td}{{\tilde{d}}}
\nc{\te}{{\tilde{e}}}
\nc{\tf}{{\tilde{f}}}
\nc{\tg}{{\tilde{g}}}
\nc{\ti}{{\tilde{i}}}
\nc{\tj}{{\tilde{j}}}
\nc{\tk}{{\tilde{k}}}
\nc{\tl}{{\tilde{l}}}
\nc{\tm}{{\tilde{m}}}
\nc{\tn}{{\tilde{n}}}
\nc{\tp}{{\tilde{p}}}
\nc{\tq}{{\tilde{q}}}
\nc{\tr}{{\tilde{r}}}
\nc{\ts}{{\tilde{s}}}
\nc{\tu}{{\tilde{u}}}
\nc{\tv}{{\tilde{v}}}
\nc{\tw}{{\tilde{w}}}
\nc{\tx}{{\tilde{x}}}
\nc{\ty}{{\tilde{y}}}
\nc{\tz}{{\tilde{z}}}

\nc{\hA}{{\hat{A}}}
\nc{\hB}{{\hat{B}}}
\nc{\hC}{{\hat{C}}}
\nc{\hD}{{\hat{D}}}
\nc{\hE}{{\hat{E}}}
\nc{\hF}{{\hat{\rF}}}
\nc{\hG}{{\hat{G}}}
\nc{\hH}{{\hat{H}}}
\nc{\hI}{{\hat{I}}}
\nc{\hJ}{{\hat{J}}}
\nc{\hK}{{\hat{K}}}
\nc{\hL}{{\hat{L}}}
\nc{\hM}{{\hat{M}}}
\nc{\hN}{{\hat{N}}}
\nc{\hO}{{\hat{O}}}
\nc{\hP}{{\hat{P}}}
\nc{\hQ}{{\hat{Q}}}
\nc{\hR}{{\hat{R}}}
\nc{\hS}{{\hat{S}}}
\nc{\hT}{{\hat{T}}}
\nc{\hU}{{\hat{U}}}
\nc{\hV}{{\hat{V}}}
\nc{\hW}{{\hat{W}}}
\nc{\hX}{{\hat{X}}}
\nc{\hY}{{\hat{Y}}}
\nc{\hZ}{{\hat{Z}}}

\nc{\ha}{{\hat{a}}}
\nc{\hb}{{\hat{b}}}
\nc{\hc}{{\hat{c}}}
\nc{\hd}{{\hat{d}}}
\nc{\he}{{\hat{e}}}
\nc{\hf}{{\hat{f}}}
\nc{\hg}{{\hat{g}}}
\nc{\hh}{{\hat{h}}}
\nc{\hi}{{\hat{i}}}
\nc{\hj}{{\hat{j}}}
\nc{\hk}{{\hat{k}}}
\nc{\hl}{{\hat{l}}}
\nc{\hm}{{\hat{m}}}
\nc{\hn}{{\hat{n}}}
\nc{\ho}{{\hat{o}}}
\nc{\hp}{{\hat{p}}}
\nc{\hq}{{\hat{q}}}
\nc{\hr}{{\hat{r}}}
\nc{\hs}{{\hat{s}}}
\nc{\hu}{{\hat{u}}}
\nc{\hv}{{\hat{v}}}
\nc{\hw}{{\hat{w}}}
\nc{\hx}{{\hat{x}}}
\nc{\hy}{{\hat{y}}}
\nc{\hz}{{\hat{z}}}

\nc{\rF}{{\mathrm{F}}}
\nc{\rG}{{\mathrm{G}}}
\nc{\rQ}{{\mathrm{Q}}}

\nc{\eps}{\varepsilon}
\nc{\lan}{\big\langle}
\nc{\ran}{\big\rangle}
\nc{\kk}{{\mathsf{k}}}

\renewcommand{\P}{\PP}

\def\bw#1#2{\textstyle{\bigwedge\hskip-0.9mm^{#1}}\hskip0.2mm{#2}}

\DeclareMathOperator{\Ext}{\mathrm{Ext}}

\DeclareMathOperator{\Pic}{\mathrm{Pic}}

\DeclareMathOperator{\Gr}{\mathrm{Gr}}

\theoremstyle{plain}

\newtheorem{theorem}{Theorem}

\newtheorem{lemma}[theorem]{Lemma}
\newtheorem{proposition}[theorem]{Proposition}
\newtheorem{corollary}[theorem]{Corollary}

\theoremstyle{definition}

\theoremstyle{remark}

\newtheorem{remark}[theorem]{Remark}

\title[Derived equivalence of Ito--Miura--Okawa--Ueda Calabi--Yau 3-folds]{Derived equivalence of Ito--Miura--Okawa--Ueda\\[2ex]Calabi--Yau 3-folds}
\author{Alexander Kuznetsov}
\address{{\sloppy
\parbox{0.9\textwidth}{
Algebraic Geometry Section, Steklov Mathematical Institute of Russian Academy of Sciences,\\
8 Gubkin str., Moscow 119991 Russia
\\[5pt]
The Poncelet Laboratory, Independent University of Moscow
\hfill\\[5pt]
Laboratory of Algebraic Geometry, National Research University Higher School of Economics, Russian Federation
}\bigskip}}
\email{akuznet@mi.ras.ru}
\date{}
\thanks{I was partially supported by the Russian Academic Excellence Project ``5-100'', by RFBR grants 14-01-00416, 15-01-02164, 15-51-50045, and by the Simons foundation.}

\begin{document}

\begin{abstract}
We prove derived equivalence of Calabi--Yau threefolds constructed by Ito--Miura--Okawa--Ueda as an example of non-birational Calabi--Yau varieties 
whose difference in the Grothendieck ring of varieties is annihilated by the affine line.
\end{abstract}

\maketitle


In a recent paper~\cite{IMOU} there was constructed a pair of Calabi--Yau threefolds $X$ and $Y$ such that their classes $[X]$ and $[Y]$ in the Grothendieck group of varieties are different, but
\begin{equation*}
([X] - [Y])[\AA^1] = 0.
\end{equation*}
The goal of this short note is to show that these threefolds are derived equivalent
\begin{equation*}
\bD(X) \cong \bD(Y).
\end{equation*}
In course of proof we will construct an explicit equivalence of the categories.

We denote by $\kk$ the base field. 
All the functors between triangulated categories are implicitly derived.

\medskip

As explained in~\cite{IMOU} the threefolds $X$ and $Y$ are related by the following diagram
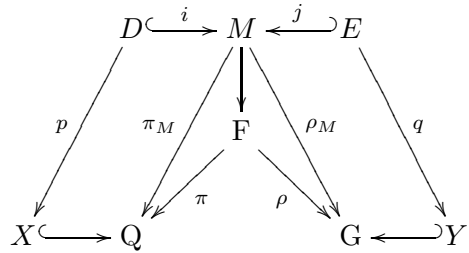
\begin{equation*}
\xymatrix{
& D \ar@{^{(}->}[r]^-i \ar[ddl]_p & M \ar[d] \ar[ddl]_{\pi_M} \ar[ddr]^{\rho_M} & E \ar@{_{(}->}[l]_-j \ar[ddr]^q \\
&& \rF \ar[dl]^\pi \ar[dr]_\rho \\
X  \ar@{^{(}->}[r] & \rQ && \rG & Y \ar@{_{(}->}[l] 
}
\end{equation*}
Here 
\begin{itemize}
\item 
$\rF$ is the flag variety of the simple algebraic group of type $\sG_2$, 
\item 
$\rQ$ and $\rG$ are the Grassmannians of this group:
\begin{itemize}
\item 
$\rQ$ is a 5-dimensional quadric in $\P(V)$, where $V$ is the 7-dimensional fundamental representation, and 
\item 
$\rG = \Gr(2,V) \cap \P(W)$, where $W \subset \bw2V$ is the 14-dimensional adjoint representation
(this intersection is not dimensionally transverse!),
\end{itemize}
\item 
$\pi \colon \rF \to \rQ$ and $\rho  \colon \rF \to \rG$ are Zariski locally trivial $\P^1$-fibrations,
\item 
$M$ is a smooth half-anticanonical divisor in $\rF$,
\item 
$\pi_M := \pi\vert_M \colon M \to \rQ$ is the blowup with center in the Calabi--Yau threefold $X$,
\item 
$\rho_M := \rho\vert_M \colon M \to \rG$ is the blowup with center in the Calabi--Yau threefold $Y$,
\item 
$D$ and $E$ are the exceptional divisors of the blowups,
\item 
$p := \pi\vert_D \colon D \to X$ and $q := \rho\vert_E \colon E \to Y$ are the contractions.
\end{itemize}
We denote by $h$ and $H$ the hyperplane classes of $\rQ$ and $\rG$, as well as their pullbacks to $\rF$ and $M$.
Then $h$ and $H$ form a basis of $\Pic(\rF)$ in which the canonical classes can be expressed as follows:
\begin{equation}\label{eq:k}
K_\rQ = -5h,
\qquad 
K_\rG = -3H,
\qquad
K_\rF = -2H - 2h,
\qquad
K_M = -H - h.
\end{equation} 
The classes $h$ and $H$ are relative hyperplane classes for the $\P^1$-fibrations $\rho \colon \rF \to \rG$ and $\pi \colon \rF \to \rQ$ respectively.
We define rank 2 vector bundles $\cK$ and $\cU$ on $\rQ$ and $\rG$ respectively by
\begin{equation}\label{eq:pushforwards-h-H}
\pi_*\cO_F(H) \cong \cK^\vee,
\qquad 
\rho_*\cO_F(h) \cong \cU^\vee.
\end{equation} 
Then
\begin{equation*}
\P_\rQ(\cK) \cong \rF \cong \P_\rG(\cU).
\end{equation*}

It follows from~\eqref{eq:pushforwards-h-H} that $X \subset \rQ$ is the zero locus of a section of the vector bundle $\cK^\vee(h)$ on $\rQ$ and $Y \subset \rG$ is the zero locus of a section of the vector bundle $\cU^\vee(H)$ on $\rG$.


Since $H$ and $h$ are relative hyperplane classes for $\rF = \P_\rQ(\cK)$ and $\rF = \P_\rG(\cU)$ respectively, we have on $\rF$ exact sequences
\begin{equation*}
0 \to \omega_{\rF/\rQ} \to \cK^\vee(-H) \to \cO_\rF \to 0,
\qquad
0 \to \omega_{\rF/\rG} \to \cU^\vee(-h) \to \cO_\rF \to 0.
\end{equation*} 
By~\eqref{eq:k} we have $\omega_{\rF/\rQ} \cong \cO_\rF(3h-2H)$ and $\omega_{\rF/\rG} \cong \cO_\rF(H-2h)$. 
Taking the determinants of the above sequences and dualizing, we deduce 
\begin{equation}\label{eq:c1-ku}
\det(\cK) \cong \cO_\rQ(-3h),
\qquad 
\det(\cU) \cong \cO_\rG(-H).
\end{equation}
Furthermore, twisting the sequences by $\cO_\rF(H)$ and $\cO_\rF(h)$ respectively, we obtain 
\begin{equation}\label{eq:seq-k}
0 \to \cO_\rF(3h-H) \to \cK^\vee \to \cO_\rF(H) \to 0,
\end{equation} 
and
\begin{equation}\label{eq:seq-u}
0 \to \cO_\rF(H-h) \to \cU^\vee \to \cO_\rF(h) \to 0.
\end{equation}

Derived categories of both $\rQ$ and $\rG$ are known to be generated by exceptional collections.
In fact, for our purposes the most convenient collections are 
\begin{equation}\label{eq:d-q}
\bD(\rQ) = \langle \cO_\rQ(-3h), \cO_\rQ(-2h), \cO_\rQ(-h), \cS, \cO_\rQ, \cO_\rQ(h) \rangle,
\end{equation}
where $\cS$ is the spinor vector bundle of rank 4, see~\cite{Kap}, and
\begin{equation}\label{eq:d-g}
\bD(\rG) = \langle \cO_\rG(-H),\cU,\cO_\rG,\cU^\vee,\cO_\rG(H),\cU^\vee(H) \rangle.
\end{equation} 
This collection is obtained from the collection of~\cite[Section~6.4]{K} by a twist (note that $\cU \cong \cU^\vee(-H)$ by~\eqref{eq:c1-ku}).
In fact, for the argument below one even does not need to know that this exceptional collection is full; on a contrary, one can use the argument to prove its fullness, see Remark~\ref{remark:full}.

Using two blowup representations of $M$ and the corresponding semiorthogonal decompositions 
\begin{equation}\label{eq:d-m}
\langle \pi_M^*(\bD(\rQ)), i_*p^*(\bD(X)) \rangle = \bD(M) = \langle \rho_M^*(\bD(\rG)), j_*q^*(\bD(Y) \rangle
\end{equation}
together with the above exceptional collections, we see that $\bD(X)$ and $\bD(Y)$ are the complements in~$\bD(M)$ of exceptional collections of length 6, so one can guess they are equivalent.
Below we show that this is the case by constructing a sequence of mutations transforming one exceptional collection to the other.

We start with some cohomology computations:

\begin{lemma}\label{lemma:cohomology-f}
$(i)$ Line bundles $\cO_\rF(th-H)$ and $\cO_\rF(tH-h)$ are acyclic for all $t \in \ZZ$.

\noindent$(ii)$ Line bundles $\cO_\rF(-2H)$ and $\cO_\rF(2h-2H)$ are acyclic and 
\begin{equation*}
H^\bullet(\rF,\cO_\rF(3h-2H)) = \kk[-1].
\end{equation*}

\noindent$(iii)$ Vector bundles $\cU(-2H)$, $\cU(-H)$, $\cU(h-H)$, and $\cU \otimes \cU(-H)$ on $\rF$ are acyclic and
%
\begin{equation*}
H^\bullet(\rF,\cU(h)) = \kk,
\qquad 
H^\bullet(\rF,\cU \otimes \cU(h)) \cong \kk[-1].
\end{equation*}
\end{lemma}
\begin{proof}
Part $(i)$ is easy since $\pi_*\cO_\rF(-H) = 0$ and $\rho_*\cO_\rF(-h) = 0$.
For part $(ii)$ we note that
\begin{equation}\label{eq:pi-2H}
\pi_*\cO_\rF(-2H) \cong (\det\cK)[-1] \cong \cO_\rQ(-3h)[-1], 
\end{equation} 
so acyclicity of $\cO_\rF(th-2H)$ for $-1 \le t \le 2$ and the formula for the cohomology of $\cO_\rF(3h-2H)$ follow.
For part $(iii)$ we push forward the bundles $\cU(-2H)$, $\cU(-H)$, $\cU(h-H)$, and $\cU \otimes \cU(-H)$ to $\rG$ and applying~\eqref{eq:pushforwards-h-H} we obtain 
\begin{equation*}
\cU(-2H),\ \cU(-H),\ \cU \otimes \cU^\vee(-H),\ \cU \otimes \cU(-H).
\end{equation*}
Their acyclicity follows from orthogonality of $\cU^\vee(H)$ to the collection $(\cO_\rG(-H),\cU,\cO_\rQ,\cU^\vee)$ in view of the exceptional collection~\eqref{eq:d-g}.
Analogously,  pushing forward $\cU(h)$ to $\rG$ we obtain $\cU \otimes \cU^\vee$, and its cohomology is $\kk$ since $\cU$ is exceptional.
Finally, using~\eqref{eq:seq-u} we see that $\cU \otimes \cU(h)$ has a filtration with factors $\cO_\rF(-h)$, $\cO_\rF(h-H)$, and $\cO_\rF(3h-2H)$.
The first two are acyclic by part~$(i)$ and the last one has cohomology~$\kk[-1]$ by part~$(ii)$.
It follows that the cohomology of $\cU \otimes \cU(h)$ is also $\kk[-1]$.
\end{proof}

\begin{corollary}\label{corollary:cohomology-m}
The following line and vector bundles are acyclic on $M$:
\begin{equation*}
\cO_M(h-H),\ \cO_M(3h-H),\ \cU(h-H).
\end{equation*}
Moreover,
\begin{equation*}
H^\bullet(M,\cU(h)) = \kk,
\qquad 
H^\bullet(M,\cU \otimes \cU(h)) = \kk[-1].
\end{equation*}
\end{corollary}
\begin{proof}
Since $M \subset \rF$ is a divisor with class $h + H$ we have a resolution
\begin{equation*}
0 \to \cO_\rF(-h-H) \to \cO_\rF \to \cO_M \to 0.
\end{equation*}
Tensoring it with the required bundles and using the Lemma~\ref{lemma:cohomology-f} we obtain the required results.
\end{proof}

\begin{proposition}
We have an exact sequence on $\rF$ and $M$:
\begin{equation}\label{eq:seq-s}
0 \to \cU \to  \cS' \to \cU^\vee(-h) \to  0,
\end{equation} 
where $\cS'$ is \textup{(}the pullback to $\rF$ or $M$ of\textup{)} a rank $4$ vector bundle on $\rQ$.
\end{proposition}

Later we will identify the bundle $\cS'$ constructed as extension~\eqref{eq:seq-s} with the spinor bundle $\cS$ on~$\rQ$.

\begin{proof}
We will construct this exact sequence on $\rF$, and then restrict it to $M$.
First, note that by Lemma~\ref{lemma:cohomology-f} we have $\Ext^\bullet(\cU^\vee(-h),\cU) \cong H^\bullet(\rF,\cU \otimes \cU(h)) \cong \kk[-1]$, hence there is a canonical extension of $\cU^\vee(-h)$ by $\cU$.
We denote by $\cS'$ the extension, so that we have an exact sequence~\eqref{eq:seq-s}.
Obviously, $\cS'$ is locally free of rank~4.
We have to check that it is a pullback from $\rQ$.

Using exact sequences
\begin{equation*}
0 \to \cO_\rF(-h) \to \cU \to \cO_\rF(h-H) \to 0
\qquad\text{and}\qquad
0 \to \cO_\rF(H-2h) \to \cU^\vee(-h) \to \cO_\rF \to 0
\end{equation*}
(obtained from~\eqref{eq:seq-u} by the dualization and a twist) and the cohomology computations of Lemma~\ref{lemma:cohomology-f}, 
we see that extension~\eqref{eq:seq-s} is induced by a class in $\Ext^1(\cO_\rF(H-2h),\cO_\rF(h-H)) \cong H^\bullet(\rF,\cO_\rF(3h-2H) = \kk[-1]$.
By~\eqref{eq:seq-k} the corresponding extension is~$\cK^\vee(-2h)$.
It follows that the sheaf $\cS'$ has a 3-step filtration with factors being $\cO_\rF(-h)$, $\cK^\vee(-h)$, and~$\cO_\rF$.
All these sheaves are pullbacks from $\rQ$, and since the subcategory $\pi^*(\bD(\rQ)) \subset \bD(\rF)$ is triangulated (because the functor $\pi^*$ is fully faithful), it follows that~$\cS'$ is also a pullback from $\rQ$.
\end{proof}


\medskip

Now we are ready to explain the mutations.
We start with a semiorthogonal decomposition
\begin{equation}\label{step0}
\bD(M) = \langle \cO_M(-H),\cU,\cO_M,\cU^\vee,\cO_M(H),\cU^\vee(H), \Phi_0(\bD(Y)) \rangle,
\end{equation}
\begin{equation}\label{eq:phi0}
\Phi_0 = j_*\circ q^* \colon \bD(Y) \to \bD(M),
\end{equation} 
obtained by plugging~\eqref{eq:d-g} into the right hand side of~\eqref{eq:d-m}. 
Now we apply a sequence of mutations, modifying the functor $\Phi_0$.

First, we mutate $\Phi_0(\bD(Y))$ two steps to the left:
\begin{equation}\label{step1}
\bD(M) = \langle \cO_M(-H),\cU,\cO_M,\cU^\vee,\Phi_1(\bD(Y)),\cO_M(H),\cU^\vee(H) \rangle,
\end{equation}
\begin{equation}\label{eq:phi1}
\Phi_1 = \bL_{\langle \cO_M(H), \cU^\vee(H) \rangle} \circ \Phi_0.
\end{equation} 
Here $\bL$ denotes the left mutation functor.

Next, we mutate the last two terms to the far left (these objects got twisted by $K_M = -h - H$):
\begin{equation*}\label{step2}
\bD(M) = \langle \cO_M(-h),\cU^\vee(-h),\cO_M(-H),\cU,\cO_M,\cU^\vee,\Phi_1(\bD(Y)) \rangle.
\end{equation*}

Next, we mutate $\cO_M(-h)$ and $\cU^\vee(-h)$ one step to the right. 
As $\Ext^\bullet(\cU^\vee(-h),\cO_M(-H)) \cong H^\bullet(M,\cU(h-H)) = 0$,
and $\Ext^\bullet(\cO_M(-h),\cO_M(-H)) \cong H^\bullet(M,\cO_M(h-H)) = 0$ by Corollary~\ref{corollary:cohomology-m}, we obtain
\begin{equation*}\label{step3}
\bD(M) = \langle \cO_M(-H),\cO_M(-h),\cU^\vee(-h),\cU,\cO_M,\cU^\vee,\Phi_1(\bD(Y)) \rangle.
\end{equation*}

Next, we mutate $\cU$ one step to the left. 
As $\Ext^\bullet(\cU^\vee(-h),\cU) \cong H^\bullet(\cU \otimes \cU(h)) \cong \kk[-1]$ by Corollary~\ref{corollary:cohomology-m}, the resulting mutation is an extension, which in view of~\eqref{eq:seq-s} gives $\cS'$.
Thus, we obtain
\begin{equation*}\label{step4}
\bD(M) = \langle \cO_M(-H),\cO_M(-h),\cS',\cU^\vee(-h),\cO_M,\cU^\vee,\Phi_1(\bD(Y)) \rangle.
\end{equation*}


Next, we mutate $\cO_M(-H)$ to the far right (this object got twisted by $-K_M = h + H$):
\begin{equation*}\label{step6}
\bD(M) = \langle \cO_M(-h),\cS',\cU^\vee(-h),\cO_M,\cU^\vee,\Phi_1(\bD(Y)),\cO_M(h) \rangle.
\end{equation*}

Next, we mutate $\Phi_1(\bD(Y))$ one step to the right:
\begin{equation*}\label{step7}
\bD(M) = \langle \cO_M(-h),\cS',\cU^\vee(-h),\cO_M,\cU^\vee,\cO_M(h),\Phi_2(\bD(Y)) \rangle, 
\end{equation*}
\begin{equation}\label{eq:phi2}
\Phi_2 = \bR_{\cO_M(h)} \circ \Phi_1.
\end{equation} 
Here $\bR$ denotes the right mutation functor.

Next, we mutate simultaneously $\cU^\vee(-h)$ and $\cU^\vee$ one step to the right. 
As $\Ext^\bullet(\cU^\vee(-h),\cO_M) \cong \Ext^\bullet(\cU^\vee,\cO_M(h)) = H^\bullet(M,\cU(h)) = \kk$ by Corollary~\ref{corollary:cohomology-m}, the resulting mutation is the cone of a morphism, 
which in view of~\eqref{eq:seq-u} and its twist by $\cO_M(-h)$ gives $\cO_M(H-2h)$ and $\cO_M(H-h)$ respectively.
Thus we obtain
\begin{equation*}\label{step8}
\bD(M) = \langle \cO_M(-h),\cS',\cO_M,\cO_M(H-2h),\cO_M(h),\cO_M(H-h),\Phi_2(\bD(Y)) \rangle.
\end{equation*}

Next, we mutate $\cO_M(h)$ one step to the left.
As $\Ext^\bullet(\cO_M(H-2h),\cO_M(h)) \cong H^\bullet(M,\cO_M(3h-H)) = 0$ by Corollary~\ref{corollary:cohomology-m}, we obtain
\begin{equation*}\label{step9}
\bD(M) = \langle \cO_M(-h),\cS',\cO_M,\cO_M(h),\cO_M(H-2h),\cO_M(H-h),\Phi_2(\bD(Y)) \rangle.
\end{equation*}

Next, we mutate $\Phi_2(\bD(Y))$ two steps to the left:
\begin{equation*}\label{step10}
\bD(M) = \langle \cO_M(-h),\cS',\cO_M,\cO_M(h),\Phi_3(\bD(Y)),\cO_M(H-2h),\cO_M(H-h) \rangle,
\end{equation*}
\begin{equation}\label{eq:phi3}
\Phi_3 = \bL_{\langle \cO_M(H-2h), \cO_M(H-h) \rangle} \circ \Phi_2.
\end{equation}

Finally, we mutate $\cO_M(H-2h)$ and $\cO_M(H-h)$ to the far left:
\begin{equation}\label{step11}
\bD(M) = \langle \cO_M(-3h),\cO_M(-2h),\cO_M(-h),\cS',\cO_M,\cO_M(h),\Phi_3(\bD(Y)) \rangle.
\end{equation}

\medskip

Now we finished with mutations, and it remains to check that the resulting semiorthogonal decomposition provides an equivalence of categories.
To do this, we first observe the following

\begin{lemma}
The bundle $\cS'$ is isomorphic to the spinor bundle $\cS$ on $\rQ$.
\end{lemma}
\begin{proof}
The first six objects in~\eqref{step11} are pullbacks from $\rQ$ by $\pi_M$.
Since $\pi_M^*$ is fully faithful, the corresponding objects on $\rQ$ are also semiorthogonal.
In particular, the bundle $\cS'$ on $\rQ$ is right orthogonal to~$\cO_\rQ$ and $\cO_\rQ(h)$ and left orthogonal to $\cO_\rQ(-3h)$, $\cO_\rQ(-2h)$, and $\cO_\rQ(-h)$.
By~\eqref{eq:d-q} the intersection of these orthogonals is generated by the spinor bundle~$\cS$.
Therefore, $\cS'$ is a multiple of the spinor bundle~$\cS$.
Since the ranks of both $\cS'$ and $\cS$ are 4, the multiplicity is 1, so $\cS' \cong \cS$.
\end{proof}


Thus the first six objects of~\eqref{step11} generate $\pi_M^*(\bD(\rQ))$. 
Comparing~\eqref{step11} with~\eqref{eq:d-q} and~\eqref{eq:d-m}, we conclude that the last component $\Phi_3(\bD(Y))$ coincides with $i_*q^*(\bD(X))$. 
Altogether, this proves the following

\begin{theorem}
The functor 
\begin{equation*}
\Phi_3 = \bL_{\langle \cO(H-2h), \cO(H-h) \rangle} \circ \bR_{\cO(h)} \circ \bL_{\langle \cO(H), \cU^\vee(H) \rangle} \circ j_* \circ q^* \colon \bD(Y) \to \bD(M)
\end{equation*}
is an equivalence of $\bD(Y)$ onto the triangulated subcategory of $\bD(M)$ equivalent to $\bD(X)$ via the embedding $i_* \circ p^* \colon \bD(X) \to \bD(M)$.
In particular, the functor
\begin{equation*}
\Psi = p_* \circ i^! \circ \bL_{\langle \cO(H-2h), \cO(H-h) \rangle} \circ \bR_{\cO(h)} \circ \bL_{\langle \cO(H), \cU^\vee(H) \rangle} \circ j_* \circ q^* \colon \bD(Y) \to \bD(X)
\end{equation*}
is an equivalence of categories.
\end{theorem}

\begin{remark}\label{remark:full}
Let us sketch how the arguments above can be also used to prove fullness of~\eqref{eq:d-g}.
Denote by~$\cC$ the orthogonal to the collection~\eqref{eq:d-g} in $\bD(\rG)$.
Then we still have a semiorthogonal decomposition~\eqref{step0}, with~$\Phi_0(\bD(Y))$ replaced by $\langle \cC, \Phi_0(\bD(Y)) \rangle$.
We can perform the same sequence of mutation, keeping the subcategory $\cC$ together with $\bD(Y)$.
For instance, in~\eqref{step1} we write $\langle \bL_{\langle \cO_M(H), \cU^\vee(H) \rangle}(\cC), \Phi_1(\bD(Y)) \rangle$ instead of just $\Phi_1(\bD(Y))$ and so on.
In the end, we arrive at~\eqref{step11} with $\Phi_3(\bD(Y))$ replaced by $\langle \cC', \Phi_3(\bD(Y)) \rangle$ with $\cC'$ equivalent to $\cC$.
Comparing it with~\eqref{eq:d-q} and~\eqref{eq:d-m}, we deduce that $\bD(X)$ has a semiorthogonal decomposition with two components equivalent to $\cC$ and $\bD(Y)$.
But $X$ is a Calabi--Yau variety, hence its derived category has no nontrivial semiorthogonal decompositions by~\cite{Bri}.
Therefore $\cC = 0$ and so exceptional collection~\eqref{eq:d-g} is full.
\end{remark}


\noindent{\bf Acknowledgement:} 
I would like to thank Shinnosuke Okawa and Evgeny Shinder for their comments on the first draft of the paper.

\end{document}